\documentclass[12pt]{article}
\usepackage{amsmath,amsthm,amssymb}
\usepackage{url,xspace}
\usepackage{enumerate}
\usepackage{url}
\usepackage{hyperref}
\hypersetup{citecolor=red, linkcolor=blue, colorlinks=true}

\newtheorem{theorem}{Theorem}[section]
\newtheorem{corollary}[theorem]{Corollary}
\newtheorem{proposition}[theorem]{Proposition}

\newtheorem{lemma}[theorem]{Lemma}

\theoremstyle{definition}
\newtheorem{definition}[theorem]{Definition}
\newtheorem{remark}[theorem]{Remark}
\newtheorem{example}[theorem]{Example}

\newcommand{\F}{\mathbb{F}}

\newcommand{\Aut}{\mathrm{Aut}}

\newcommand{\But}{\textrm{BH}}

\newcommand\blfootnote[1]{%
  \begingroup
  \renewcommand\thefootnote{}\footnote{#1}%
  \addtocounter{footnote}{-1}%
  \endgroup
}

\newcommand\nth{\textsuperscript{th}\xspace}

\title{\textbf{\Large{Spectra of Hadamard matrices}}}

\author{
\textsc{Ronan Egan}
				\thanks{\textit{E-mail: ronan.egan@nuigalway.ie}}\\
\textit{\footnotesize{School of Mathematics, Statistics and Applied Mathematics}}\\
\textit{\footnotesize{National University of Ireland, Galway}}\\
\textit{\footnotesize{Galway, Ireland}}\\
\textsc{Padraig \'O Cath\'ain}
				\thanks{\textit{E-mail: pocathain@wpi.edu}}\\
\textit{\footnotesize{Department of Mathematical Sciences}}\\
\textit{\footnotesize{Worcester Polytechnic Institute}}\\
\textit{\footnotesize{Worcester, MA, USA}}\\
\textsc{Eric Swartz}
				\thanks{\textit{E-mail: easwartz@wm.edu}}\\
\textit{\footnotesize{Department of Mathematics}}\\
\textit{\footnotesize{College of William \& Mary}}\\
\textit{\footnotesize{Williamsburg, VA, USA}}}

\begin{document}
\thispagestyle{empty}
\maketitle

\begin{abstract}
In this paper, we develop a technique for controlling the spectra of Hadamard matrices with sufficiently rich automorphism groups. For each integer $t \geq 2$, we construct a Hadamard matrix $H_{t}$ equivalent to the Sylvester matrix of order $n_{t} = 2^{2^{t-1}-1}$ such that the minimal polynomial of $\frac{1}{\sqrt{n_{t}}}H_{t}$ is the cyclotomic polynomial $\Phi_{2^{t+1}}(x)$. As an application we construct real Hadamard matrices from Butson Hadamard matrices.

More concretely, a Butson Hadamard matrix $H$ has entries in the $k^{\nth}$ roots of unity and satisfies the matrix equation $HH^{\ast} = nI_{n}$. We write $\mathrm{BH}(n, k)$ for the set of such matrices. A \textit{complete morphism} of Butson matrices is a map $\mathrm{BH}(n, k) \rightarrow \mathrm{BH}(m, \ell)$.  The matrices $H_{t}$ yield new examples of complete morphisms
\[ \mathrm{BH}(n, 2^{t}) \rightarrow \mathrm{BH}(2^{2^{t-1}-1}n, 2)\,, \]
for each $t \geq 2$, generalising a well-known result of Turyn.


\end{abstract}

\section{Introduction}

Let $M$ be an $n \times n$ matrix with entries in the complex numbers $\mathbb{C}$. If every entry $m_{ij}$ of $M$ has modulus bounded by $1$,
then Hadamard's theorem states that $|\det(M)| \leq n^{n/2}$. Hadamard himself observed that a matrix $M$ meets this bound with equality if and
only if every entry in $M$ has modulus $1$, and every pair of distinct rows of $M$ are orthogonal (with respect to the usual Hermitian inner product) \cite{Hadamard1893}.  Such a matrix is said to be \emph{Hadamard}, though the term has become synonymous with the special case where entries are in $\{\pm 1\}$.  If $M$ is Hadamard with entries in the $k\nth$ roots of unity for some $k$, then $M$ is \emph{Butson Hadamard} $\But(n,k)$ or just \emph{Butson}, named for their appearance in a paper of Butson \cite{Butson}.

In \cite{mypaper-morphisms}, the authors define a \textit{(complete) morphism} of Butson matrices to be a function from $\But(n,k)$ to $\But(r,\ell)$.  A \textit{partial morphism} is a morphism such that the domain is a proper subset of $\But(n,k)$.  We will restrict our attention to morphisms which come from embeddings of matrix algebras (as the tensor product does); such morphisms can be considered generalized plug-in constructions.
We recall some key definitions.

For each positive integer $k$, we define $\zeta_{k} = e^{2\pi i/k}$ and set $G_{k} = \langle \zeta_{k} \rangle$.  We define $H^{\phi}$ to be the entrywise application of $\phi$ to $H$ whenever $\phi$ is a function defined on the entries of $H$, and for each integer $r$ we write $H^{(r)}$ for the function which replaces each entry of $H$ by its $r\nth$ power.  This should be distinguished from the notation $H^r$ meaning the $r\nth$ power of $H$ under the usual matrix product.

\begin{definition} \label{defn:sound}
Let $X,Y \subseteq G_{k}$ be fixed.
Suppose that $H \in \But(n,k)$ is such that every entry of $H$ is contained in $X$,
and that $M \in \But(m,\ell)$ is such that every eigenvalue of $\sqrt{m}^{-1}M$ is contained in $Y$.

Then the pair $(H, M)$ is $(X,Y)$-\textit{sound} if
\begin{enumerate}
\item For each $\zeta_{k}^{j} \in X$, we have $\sqrt{m}^{1-j} M^{j} \in \But(m, \ell)$.
\item For each $\zeta_{k}^{j} \in Y$, we have $H^{(j)} \in \But(n,k)$.
\end{enumerate}
We will often say that $(H,M)$ is a \textit{sound pair} if there exist sets $X$ and $Y$ for which $(H, M)$ is $(X, Y)$-sound.
\end{definition}

\begin{theorem}[Theorem 4 \cite{mypaper-morphisms}]
\label{thm:main}
Let $H \in \But(n,k)$ and $M \in \But(m,\ell)$ be Hadamard matrices.
Define a map $\phi: \zeta_{k}^{j} \mapsto \sqrt{m}^{1-j} M^{j}$, and write
$H^{\phi}$ for the entrywise application of $\phi$ to $H$.
If $(H, M)$ is a sound pair then $H^{\phi} \in \But(mn,\ell)$.
\end{theorem}

\begin{example}
Let $H$ be any $\But(n,4)$ and let
\[ M = \left[ \begin{array}{rr} 1 & 1 \\ -1 & 1 \end{array}\right].\]
Then $(\zeta_{8}H,M)$ is $(X,Y)$-sound where $X = \{\zeta_{8},\zeta_{8}^{3},\zeta_{8}^{5},\zeta_{8}^{7}\}$ and $Y = \{\zeta_{8},\zeta_{8}^{7}\}$.  By Theorem \ref{thm:main}, from any $H \in \But(n,4)$ we can construct a real Hadamard matrix of order $2n$.  This example is due to Turyn \cite{Turyn}, though see also \cite{Cohn65}.
\end{example}

Adhering to the notation of Definition \ref{defn:sound}, if the set $Y$ of eigenvalues of $\sqrt{m}^{-1}M$ contains only primitive $k\nth$ roots of unity, the second condition is vacuous. If in addition, $p^{\alpha+1}$ divides $\varphi(k)$, where $\varphi$ denotes the Euler phi function and $p$ is a prime, the primitive $k\nth$ roots contain a translate of the roots of unity of order $p^{\alpha}$, from which we can construct a complete morphism. The problem of constructing complete morphisms is the motivation of this paper. We expect that such complete morphisms could be used both to give new proofs of known constructions of Hadamard matrices and to construct Hadamard matrices at previously unknown orders.

We focus on real Hadamard matrices hereafter.  Let $\chi(M)$ and $\mathfrak{m}(M)$ denote the characteristic and minimal polynomials of a matrix $M$.  We recall that two (real) Hadamard matrices, $H$ and $H'$ are \textit{Hadamard equivalent} if there exist monomial $\{\pm 1\}$-matrices $P$ and $Q$ such that $PHQ^{\top} = H'$. Unless $P = Q$, the Hadamard matrices need not have the same spectra, but it is well known that every eigenvalue of a Hadamard matrix $H$ of order $n$ is of absolute value $\sqrt{n}$.  In the case when $H' = H$, the collection of all such ordered pairs $(P,Q)$ forms the \textit{automorphism group} of $H$, which is denoted $\Aut(H)$.  For convenience, we will generally work with the rescaled characteristic polynomial $\chi(n^{-1/2}H)$, for which
every root must be a complex number of norm $1$. In this paper we will establish the following theorem.

\begin{theorem} \label{thm:hadexist}
For each integer $t \geq 2$ there exists a real Hadamard matrix $M$ of order $n_{t} = 2^{2^{t-1}-1}$ such that $\mathfrak{m}(n_{t}^{-1/2}M)$ is the cyclotomic polynomial $\Phi_{2^{t+1}}(x)= x^{2^{t}} + 1$.
\end{theorem}

We obtain the following result directly from Theorems \ref{thm:main} and \ref{thm:hadexist}.
\begin{corollary}
\label{cor:explicit}
For each $t \in \mathbb{N}$, there exists a complete morphism
\[ \mathrm{BH}(n, 2^{t}) \rightarrow \mathrm{BH}(2^{2^{t-1}-1}n, 2) \,.\]
\end{corollary}
Equivalently, whenever there exists an order $n$ Hadamard matrix with entries in $\langle \zeta_{2^{t}}\rangle$,
there exists a real Hadamard matrix of order $2^{2^{t-1} - 1}n$.  This result generalizes a classical result of Turyn,
who (in our terminology) constructed a complete morphism $\But(n, 4) \rightarrow \But(2n, 2)$, which is to our knowledge the first example of a
nontrivial morphism of Hadamard matrices appearing in the literature \cite{Turyn}. This is the case $t = 2$ of Corollary \ref{cor:explicit}; the case $t = 3$ is computed in Example \ref{ex:2}.

The complete morphisms of Theorem \ref{thm:hadexist} can be considered representations of the cyclic groups of order $2^{t+1}$ in which all $2^{t}$ generators are represented by Hadamard matrices. This result is complementary to Gow's construction of a representation for the cyclic group of order $2^{t}+1$ in dimension $2^{t}$ where all non-identity elements are represented by normalized (complex) Hadamard matrices \cite{GowEasy}. He has constructed similar representations for the cyclic groups of even order in which every element but the unique involution is represented by a Hadamard matrix. While we do not obtain Hadamard matrices for all nontrivial elements in Theorem \ref{thm:hadexist}, we do retain complete control of the irreducible constituents of the representation: the irreducible representations which exhibit a primitive root of unity of order $2^{t+1}$ all occur with equal multiplicity, and no other representations occur.

While this paper was being finalized, the authors became aware of a recent result of {\"O}sterg{\aa}rd and Paavola \cite{OstergardButson} which gives a direct construction of morphisms  \[ \mathrm{BH}(n, 2^{t+1}) \rightarrow \mathrm{BH}(2n, 2^{t})\,, \]
for each $t \geq 2$. While their method produces morphisms which exhibit a smaller blow-up in dimension, our method generalizes easily to control the characteristic polynomial of Hadamard matrices which have sufficiently large automorphism groups.


This paper is organized as follows.  Section \ref{sect:autsyl} provides information about a certain family of group actions on Sylvester matrices, and Section \ref{sect:charhad} describes a method for computing the characteristic polynomial of the Hadamard matrix $PH$ when $H$ is symmetric Hadamard and $(P,Q)$ is a pair of monomial matrices acting suitably on $H$. Finally, Section \ref{sect:toward}
provides an explicit construction which proves Theorem \ref{thm:hadexist}.

\section{A group action on Sylvester matrices}
\label{sect:autsyl}

Recall that a pair $(P, Q)$ of $\{\pm 1\}$-monomial matrices is an \textit{automorphism} of the Hadamard matrix $H$ if and only if $PHQ^{\top} = H$.
In this section, we will describe a family of groups closely related to the automorphism groups of the Sylvester matrices.
Let $V$ be an $n$-dimensional vector space over $\mathbb{F}_{2}$ with a fixed basis. With respect to this basis, we write elements of $a, b \in V$ as column of vectors,
and define an inner product on $V$ by $\langle a,b\rangle = a^{\top}b$.
By abuse of notation, we identify $0, 1 \in \mathbb{F}_{2}$ with their pre-images in $\mathbb{Z}$ in our definition of the Sylvester matrix $\mathcal{S}_n$ of order $2^{n}$:
\begin{equation}\label{Sylvester}
\mathcal{S}_{n} := \left[ (-1)^{\langle a, b \rangle} \right]_{a, b \in V}\,.
\end{equation}
We assume that the row and column labels occur in the same order; in this case $\mathcal{S}_{n}$ is a symmetric matrix. If an explicit ordering of labels is required, we will use a lexicographic order. It is well known that $\mathcal{S}_{n}$ is the character table of an elementary abelian $2$-group and has a rich automorphism group, which
has been studied in detail in \cite{EganFlannery}.

\begin{definition}
Recall that $\mathrm{GL}_{n}(\mathbb{F}_{2})$ is the group of invertible matrices over $\mathbb{F}_{2}$.
Let $V$ be an $n$-dimensional vector space over $\mathbb{F}_{2}$ with specified basis (so that the action of $\mathrm{GL}_{n}(\mathbb{F}_{2})$
is well defined). Define the group $\mathcal{G}_{n} = (V \times V) \rtimes \mathrm{GL}_{n}(\mathbb{F}_{2})$ by the multiplication
\[ (u_{1}, v_{1}, L_{1}) \cdot (u_{2}, v_{2}, L_{2}) = (u_{1} + (L_{1}^{-1})^{\top}u_{2}, v_{1} + L_{1}v_{2}, L_{1}L_{2})\,. \]
\end{definition}

Denote by $r_{a}$ the row of $\mathcal{S}_{n}$ labeled by $a \in V$ and by $c_{b}$ the column labeled by $b$.

We define an action of $\mathcal{G}_{n}$ on the set $\{ r_{a}, -r_{a} \mid a \in V\}$ by
\begin{equation}\label{rowaction} (u,v,L) \cdot r_a = (-1)^{\langle La, u\rangle}r_{La + v}\,, \end{equation}
and an action on the set $\{c_{b}, -c_{b} \mid b \in V\}$ by
\begin{equation}\label{colaction} (u,v,L) \cdot c_{b} = (-1)^{\langle v, (L^{-1})^{\top} b \rangle} c_{(L^{-1})^{\top}b + u}\,. \end{equation}

We note that the subgroup $\{ (0, v, L) \mid v \in V, L \in \mathrm{GL}_{n}(\mathbb{F}_{2})\}$ acts as a doubly transitive permutation group on the set $\{ r_{a} \mid a \in V\}$, usually referred to as the \textit{affine general linear group} $\mathrm{AGL}_{n}(\mathbb{F}_{2})$.

\begin{proposition}\label{newaut}
In the action of $\mathcal{G}_{n}$ on $\mathcal{S}_{n}$ induced by the actions on rows and columns of Equations \eqref{rowaction} and \eqref{colaction},
we have $(u,v,L) \cdot \mathcal{S}_{n} = (-1)^{\langle u,v\rangle}\mathcal{S}_{n}$.
\end{proposition}

\begin{proof}
Under the action of the element $(u,v,L)$, up to a $\pm 1$ sign change, row $r_a$ is mapped to row $r_{La + v}$ and column $c_b$ is mapped to column $c_{(L^{-1})^\top b + u}$. Moreover, these mappings induce additional factors of $(-1)^{\langle La, u\rangle}$ and $(-1)^{\langle v, (L^{-1})^\top b \rangle}$, respectively.  Hence,
\[(u,v,L) \cdot (-1)^{\langle a,b \rangle} = (-1)^{\langle La, u\rangle + \langle v, (L^{-1})^\top b \rangle + \langle La + v,(L^{-1})^\top b + u \rangle} = (-1)^{\langle u,v\rangle} (-1)^{\langle a,b \rangle},\]
as desired.
\end{proof}

The action of $\mathcal{G}_{n}$ described in Proposition \ref{newaut} can be realized in a natural way as a group of pairs of monomial matrices acting on the rows and columns of $\mathcal{S}_{n}$; write $\Psi(\mathcal{G}_{n})$ for this representation. Following the usual convention in the literature on Hadamard matrices, the transpose of the righthand component of $\Psi(u,v,L)$ acts on the right of $\mathcal{S}_{n}$.

\begin{definition}\label{components}
Label the rows and columns of a $2^n \times 2^n$ matrix $M \in \mathrm{GL}_{2^n}(\mathbb{C})$ by the vectors of $V=\F_{2}^{n}$. Let $\rho(L) \in \mathrm{GL}_{2^n}(\mathbb{C})$ be the permutation matrix given by the action of $L \in \mathrm{GL}_{n}(\F_{2})$ on the elements of $V$.  Explicitly, the entry in row $u$ and column $v$ of $\rho(L)$ is $1$ if and only if $Lv = u$.  Similarly for each $v \in V$ we define the translation operator $T_{v} \in \mathrm{GL}_{2^n}(\mathbb{C})$ to be the permutation matrix with $1$ in row $u$ and column $u+v$ for every $u \in V$.  Finally, we define the diagonal matrix $D_{v} \in \mathrm{GL}_{2^n}(\mathbb{C})$ by $(-1)^{\langle u,v \rangle}$ in the row and column labeled by $u$ for every $u \in V$.
\end{definition}

It will be convenient in the sequel to have some notation relating the constituents $(u, v,L)$ of an element of $\mathcal{G}_{n}$ with matrices in characteristic $0$.
\begin{itemize}
    \item $\Psi(0, 0, L) = ( \rho(L), \rho((L^{-1})^{\top}))$.
    \item $\Psi(0, v, I) = ( T_{v}, D_{v})$.
    \item $\Psi(u, 0, I) = ( D_{u}, T_{u})$.
\end{itemize}

In the current work we are interested only in constructing Hadamard matrices where the characteristic polynomial is an even function; in this case we have that $\chi(H) = \chi(-H)$. As such, we do not require that $P\mathcal{S}_{n}Q^{\top} = \mathcal{S}_{n}$, rather we require only that $P\mathcal{S}_{n}Q^{\top} = \pm \mathcal{S}_{n}$.

\begin{remark}
The group $\langle (-I, I), (I, -I), \Psi(\mathcal{G}_{n})\rangle$ is of order $4|\mathcal{G}_{n}|$ and acts transitively on the set $\{\mathcal{S}_{n}, -\mathcal{S}_{n}\}$. The stabilizer of $\mathcal{S}_{n}$ under this action is a central extension of shape $2 \cdot \mathcal{G}_{n}$, and this group is in fact the full automorphism group of $\mathcal{S}_{n}$ \cite{EganFlannery}.
\end{remark}

\section{Characteristic polynomials of Hadamard matrices}
\label{sect:charhad}

The following proposition reduces the computation of characteristic polynomials of certain Hadamard matrices to computing characteristic polynomials of monomial matrices, for which there exist satisfactory techniques. This method was inspired by the elementary evaluation of the spectrum of the Discrete Fourier Transform matrix of Diaz-Vargas, Glebsky, and Rubio-Barrios \cite{FourierSpectrum}.

\begin{proposition}\label{trick}
Suppose that $H$ is a symmetric Hadamard matrix, and that $P, Q$ are monomial matrices such that $PHQ^{\top} = (-1)^{\ell}H$, where $\ell \in \{0,1\}$.
\begin{enumerate}
   \item The identity $(PH)^{2}=(-1)^{\ell}nPQ$ holds; in particular the spectra of $(PH)^{2}$ and of $(-1)^{\ell}nPQ$ are identical.
   \item The matrix $\sqrt{n}^{1-j}(PH)^j$ is Hadamard for all odd $j$.
   \item If $\mathfrak{m}(PQ) = \Phi_{2^{t}}(x)$, then $\mathfrak{m}(n^{-1/2}PH) = \Phi_{2^{t+1}}(x)$.
\end{enumerate}
\end{proposition}

\begin{proof}
\begin{enumerate}
    \item Since by definition $PHQ^{\top} = (-1)^{\ell}H$, and $Q$ is orthogonal, we have $PH = (-1)^{\ell}HQ$. So then
\[ (PH)^{2} = PH(-1)^{\ell}HQ = (-1)^{\ell}nPQ \,. \]
    \item By the above, we have that $(PH)^{2m+1} = (-1)^{\ell m}n^{m}(PQ)^m(PH)$, which up to normalisation is the product of a monomial matrix with a Hadamard matrix.
    \item Since $PH$ is Hadamard, it is normal and so diagonalisable. Suppose that $v$ is an eigenvector of $n^{-1}(PH)^{2}$ with eigenvalue $\zeta_{2^{t}}^{j}$, a root of unity of order $2^{t}$ for some positive integer $t$. Then the corresponding eigenvalue of $n^{-1/2}PH$ is a solution to the quadratic equation $x^{2} - \zeta_{2^{t}}^{j}$, which is necessarily a root of unity of order $2^{t+1}$. Since $P$ and $H$ are defined over $\mathbb{Q}$, whenever $\mathfrak{m}(n^{-1}(PH)^{2}) = \Phi_{2^{t}}(x)$, we must have $\mathfrak{m}(n^{-1/2}PH) = \Phi_{2^{t+1}}(x)$.\qedhere
\end{enumerate}
\end{proof}

We digress briefly to compute the characteristic polynomial of a monomial matrix. This material is well known, but is included for completeness.
We recall that a monomial matrix $M$ can be written uniquely in the form $M = DP$, where $D$ is a diagonal matrix and $P$ is a permutation matrix. We begin with a well-known lemma which describes the characteristic polynomial of a permutation matrix (i.e., when $D$ is trivial).

\begin{lemma} \label{permlemma}
Let $P$ be a permutation matrix with disjoint cycles $C_{1}$, $C_{2}$, \ldots, $C_{t}$ of lengths $n_{1}, \ldots, n_{t}$. Then the characteristic polynomial of $P$ is $\prod_{j = 1}^{t} \left( x^{n_{j}} - 1\right)$.
\end{lemma}

\begin{proof}
Conjugation in the symmetric group corresponds to similarity by a permutation matrix in the general linear group. It is well known that these operations preserve the characteristic polynomial. Up to conjugation in the symmetric group, we may assume that $C_{1} = (1, 2, 3, \ldots, n_{1})$. The first $n_{1}$ rows and columns of the corresponding permutation matrix
is then in Rational Canonical Form with characteristic polynomial $x^{n_{1}}-1$. Likewise, $C_{2}$ can be conjugated to the cycle $(n_{1} + 1, n_{1} + 2, \ldots, n_{1} + n_{2})$, which yields a block with characteristic polynomial $x^{n_{2}} - 1$, and the result follows by proceeding inductively.
\end{proof}

The extension to monomial matrices is routine, being a slight generalisation of the theory developed by Carter to describe the conjugacy classes of the Coxeter group of type $C_{n}$ (which consists of $\{\pm 1\}$-monomial matrices) \cite[Proposition 24]{Carter}.

\begin{proposition}\label{monomialprop}
Suppose that $M = DP$ is monomial where $D$ is diagonal and $P$ is a permutation matrix.
For each cycle $C_{j}$ of $P$, write $c_{j}$ for the product of the corresponding nonzero entries of $M$. Then the characteristic polynomial of $M$ is
\[ \prod_{j = 1}^{t} \left(x^{n_{j}} - c_{j}\right)\,.\]
\end{proposition}

\begin{proof}
Up to similarity, we may assume that $P$ is in the standard form described in Lemma \ref{permlemma}.
Consider the cycle $C_{1}$: suppose that the nonzero entries in this cycle are
$a_{1}, a_{2}, \ldots, a_{n_{1}}$. Let $F_{2}$ be the diagonal matrix with $f_{22} = a_{1}^{-1}$,
and all other entries $1$. Then $F_{2}CF_{2}^{-1}$ is similar to $C_{1}$ but has $1$ as the nonzero entry
in the first column, and $a_{1}a_{2}$ in the second. In a similar fashion, one sets $F_{3}$ to be the
matrix which differs from the identity only in that $f_{33} = a_{1}^{-1}a_{2}^{-1}$. Then
$F_{3}F_{2}C_{1}F_{2}^{-1}F_{3}^{-1}$ has two entries equal to $1$.

Proceeding in this fashion, we can set all entries but the one in the last column to be $1$. The entry in the last column is $c_{1} = \prod_{j = 1}^{n_{1}} a_{j}$; and
since the matrix is in rational canonical form, the characteristic polynomial of the block corresponding to $C_{1}$ is $x^{n_{1}} - c_{1}$. This generalizes naturally to a product of cycles.
\end{proof}

\section{Proof of main theorem}
\label{sect:toward}

Now, consider $(P, Q) = (D_{u}T_{v}\rho(A), D_{v}T_{u}\rho((A^{-1})^{\top})) \in \Psi(\mathcal{G}_{n})$ acting on $\mathcal{S}_{n}$ as described in Section \ref{sect:autsyl}. Using the multiplication operation defined on $\mathcal{G}_{n}$, we may write the product $PQ$ as follows:
\begin{equation}\label{PQ product}
PQ =  D_{u+(A^{-1})^{\top}v}T_{v + Au}\rho(A(A^{-1})^{\top})\,.
\end{equation}

Suppose that a monomial matrix $PQ$ is given, where $(P, Q) \in \Psi(\mathcal{G}_{n})$. To construct a matrix $H = P \mathcal{S}_{n}$ such that $\chi((PH)^{2}) = \chi(nPQ)$, we require three things:
\begin{enumerate}
    \item A matrix $L \in \mathrm{GL}_{n}(\mathbb{F}_{2})$ which is of the form $L = A(A^{-1})^{\top}$.
    \item A vector $v \in V$ such that $(0, v, L) \in \mathrm{AGL}_{n}(\mathbb{F}_{2})$ has all cycles of length $2^{t-1}$.
    \item Another vector $u$ such that in every cycle of the signed permutation $\Psi(u, v, L)$ the product of the nonzero elements is $-1$.
\end{enumerate}

Given an invertible matrix $A \in \mathrm{GL}_{n}(\mathbb{F}_{2})$, where $n \geq 2$, and vectors $u + (A^{-1})^{\top}v$ and $v+Au$, there always exist solutions for $u$ and $v$. So the main constraint to this method
arises from the action of the inverse-transpose map on $\mathrm{GL}_{n}(\mathbb{F}_{2})$. Luckily, this subject has been extensively studied by Gow \cite{GowInvT},
and more recently by Fulman and Guralnick \cite{FulmanGuralnick}. The following theorem is a special case of a result from the latter paper.

\begin{theorem}[Theorem 4.2, \cite{FulmanGuralnick}]\label{FulmanGuralnick}
An element of $\mathrm{GL}_{n}(\mathbb{F}_{2})$ can be written in the form $M(M^{-1})^{\top}$ if and only if
\begin{enumerate}
    \item For each eigenvalue $\lambda \neq 1$, the multiplicities of $\lambda$ and $\lambda^{-1}$ are equal.
    \item The number of Jordan blocks of each even size is even.
\end{enumerate}
\end{theorem}

To be entirely explicit, we take $n = 2^{t-1}-1$ from now on and set $L$ to be a Jordan block of maximal size in $\mathrm{GL}_{n}(\mathbb{F}_{2})$. We exhibit a decomposition of $L$ of the form $A(A^{-1})^{\top}$, as guaranteed by Theorem \ref{FulmanGuralnick}.

\begin{proposition}\label{Jordan Decomposition}
Let $L \in \mathrm{GL}_{n}(\mathbb{F}_{2})$ be a Jordan block of size $n$. Then $L = A(A^{-1})^{\top}$ where $A = A_{(t-1)}$ is the submatrix of $\otimes^{t-1} \left[\begin{array}{cc} 1 & 1 \\ 1 & 0 \end{array}\right]$, obtained by deleting the first row and last column.
\end{proposition}

\begin{proof}
It suffices to prove that $A_{(t-1)} = L(A_{(t-1)})^{\top}$.  We prove the equivalent statement that \[\left(A_{(t-1)}\right)_{i,j} = \left(A_{(t-1)}^{\top}\right)_{i,j} + \left(A_{(t-1)}^{\top}\right)_{i,j+1}\] for all $t \geq 2$, $1 \leq i \leq n-1$ and $1 \leq j \leq n$, and that the last row of $A_{(t-1)}$ and $(A_{(t-1)})^{\top}$ are identical.   Since the latter claim is a simple observation, we focus on the former.  It is simple to verify for $t = 2$. We assume then that the claim is true when $t-1=k$, and prove the general claim by induction.  Observe that
\[
A_{(k+1)} = \left[\begin{array}{c|c|c} A_{(k)} & 0_{m}^{\top} & A_{(k)} \\ \hline 1_{m} & 1 & 0_{m} \\ \hline A_{(k)} & 0_{m}^{\top} & 0_{m\times m} \end{array}\right] ~ \text{and}~ A_{(k+1)}^{\top} = \left[\begin{array}{c|c|c} A_{(k)}^{\top} & 1_{m}^{\top} & A_{(k)}^{\top} \\ \hline 0_{m} & 1 & 0_{m} \\ \hline A_{(k)}^{\top} & 0_{m}^{\top} & 0_{m\times m} \end{array}\right],
\]
\noindent where $m = 2^{k-1}-1$.  The inductive hypothesis implies we need only check that the claim holds when $j = 2^{k-1}$ or when $i \in \{2^{k-1}-1,2^{k-1}\}$, which is readily verifiable.
\end{proof}

Next, we determine the cycle type of a conjugacy class of elements of $\mathrm{AGL}_{n}(\mathbb{F}_{2})$, of the form $(0, v, L)$ where $L$ is as in Proposition \ref{Jordan Decomposition} and $v$ is in the largest generalized eigenspace of $L$; in particular we may take $v = (0,0,0, \ldots, 0, 1)^{\top}$. Our techniques are similar to those of Guest, Morris, Praeger, and Spiga \cite{PraegerAffine}. We begin by establishing the multiplicative order of $L$.

\begin{proposition}\label{JordanOrder}
Let $L$ be a Jordan block of size $n = 2^{t-1} - 1$. Then $L$ has multiplicative order $2^{t-1} = n+1$. Furthermore, $\sum_{i=0}^{n}L^{i} = 0$.
\end{proposition}

\begin{proof}
Recall that an upper unitriangular matrix belongs to the Sylow $2$-subgroup of $\mathrm{GL}_{n}(\mathbb{F}_{2})$. Write $L = I + C$ where $I$ is the identity matrix, and $C$ has $1$s on the superdiagonal but is zero elsewhere.  Then for any positive integer $x$, we have $L^{2^x} = \sum_{i=0}^{2^x}\binom{2^x}{i}(I^{2^{x}-i} C^i) = I^{2^{x}} + C^{2^x}$.  Since $C^{m}$ has a diagonal of $1$s beginning in row $m+1$, the smallest value of $x$ for which $C$ vanishes is $x=t-1$.

Setting $X = \sum_{i=0}^{n}L^i$, we clearly have that $LX = XL = X$. Row $j$ of $LX$ is the sum of row $j$ and row $j+1$ of $X$, and thus row $j+1$ of $X$ is zero for all $1 \leq j \leq n-1$.  Likewise, column $j+1$ of $XL$ is the sum of column $j$ and column $j+1$ of $X$, and so column $j$ is zero for all $1 \leq j \leq n-1$.  Finally, observe that $L^{i}_{1,n} = 0$ for all $0 \leq i \leq 2^{t-1}-3$, and that $L^{i}_{1,n} = 1$ for $i \in \{2^{t-1}-2,2^{t-1}-1\}$, and thus $X_{1,n} = 0$.
\end{proof}

We recall that the subgroup $\{(0, v, L) \mid v \in V, L \in \mathrm{GL}_{n}(\mathbb{F}_{2}) \}$ is isomorphic to the group $\mathrm{AGL}_{n}(\mathbb{F}_{2})$ and has a natural permutation action on $V$ given by $(0,v,L) \cdot a = La + v$, which extends naturally to an action on the rows and columns of $\mathcal{S}_{n}$ as described in Section \ref{sect:autsyl}.

\begin{proposition}\label{Jordanprop}
Let $v = (0,\ldots,0,1)^{\top}$ and $u = (1,0,\ldots, 0)^{\top}$. Then every orbit of $\langle (0,v,L)\rangle$ on the natural module $V$ has length $2^{t-1} = n+1$.
If $\gamma$ is the sum of elements in an orbit, then $\langle \gamma, u\rangle = 1$.
\end{proposition}

\begin{proof}
The group $\mathrm{AGL}_{n}(\mathbb{F}_{2})$ is isomorphic to a subgroup of $\mathrm{GL}_{n+1}(\mathbb{F}_{2})$ and the exponent of the Sylow $2$-subgroup of $\mathrm{GL}_{2^{t-1}}(\mathbb{F}_{2})$ is $2^{t-1}$ (see, e.g., \cite[p.~192]{Suprunenko}), thus we have by Proposition \ref{JordanOrder} that the
order of $(0, v, L) \in \mathrm{AGL}_{n}(\mathbb{F}_{2})$ is precisely $2^{t-1}$.

Let $y_{0} = 0$ and consider the orbit of $y_{0}$ under $\langle (0,v,L)\rangle$: so $y_{j} = \sum_{i=0}^{j-1}L^{i}v$ for $1 \leq j \leq n$. For each $j \geq 1$, the vector $y_{j}$ has $1$ in entry $2^{t-1} - j + 1$, and all prior entries are $0$; so the $y_{j}$ are distinct.  Hence the orbit has length $n+1$.  Define $\gamma_{0} = \sum_{j=0}^{n} y_{j}$.  Only $y_{n}$ has $1$ in its first entry, so $\langle \gamma_{0}, u\rangle = 1$, as required.

By Proposition \ref{JordanOrder} the sum $\sum_{i=0}^{n}L^{i} = 0$ where $n = 2^{t-1}-1$.  For any $u \in V$, the sum of the elements in the orbit of $u$ under $\langle (0,v,L)\rangle$ is by linearity equal to
\[\left(\sum_{i=0}^{n}L^{i}u\right) + \gamma_{0} = \gamma_{0}\,, \]
where $\gamma_{0}$ is the sum of the elements of the orbit of the zero vector, and so the result follows.
\end{proof}

Recall the notation of Definition \ref{components} and the Sylvester matrix $\mathcal{S}_{n}$ \eqref{Sylvester}, where $n = 2^{t-1}-1$.

\begin{proposition}\label{method}
Let $a = (0, 1, 1, \ldots, 1)^{\top}$ and $b = (1, 1, 1, \ldots, 1)^{\top}$ be elements of an $n$-dimensional vector space $V$ over $\mathbb{F}_{2}$.  Let $A$ be the matrix of Proposition \ref{Jordan Decomposition}.  Define $H = D_{a}T_{b}\rho(A)\mathcal{S}_{n}$.  Then $\mathfrak{m}(2^{-n/2}H) = \Phi_{2^{t+1}}(x)$.
\end{proposition}

\begin{proof}
Setting $P = D_{a}T_{b}\rho(A)$ and $Q = T_{a}D_{b}\rho((A^{-1})^{\top})$ we have that
\[ PQ = D_{u}T_{v}\rho(L)\,,\]
where $L$ is a Jordan block of size $n$ in $\mathrm{GL}_{n}(\mathbb{F}_{2})$, $v = (0,\ldots,0,1)^{\top}$, and $u = (1,0, \ldots, 0)^{\top}$.

Every orbit of $(0,v,L) \in \textrm{AGL}_{n}(\mathbb{F}_{2})$ on $V$ is of length $2^{t-1}$ by Proposition \ref{Jordanprop}. By Proposition \ref{monomialprop}, the characteristic polynomial of $PQ$ contains one factor $x^{2^{t-1}} - 1$ for each cycle in which the product of the nonzero elements is $1$, and $x^{2^{t-1}}+1$ for each cycle in which the product of elements is $-1$.

Observe that an entry in the row labeled by $\alpha$ of $PQ$ is $+1$ if and only if $\langle \alpha, u\rangle = 0$. So the product of nonzero entries in a cycle is negative if and only if the row labels of the entries in the cycle intersect the null space of the linear functional $\langle -, u\rangle$ in an odd number of points. By the second claim of Proposition \ref{Jordanprop}, every cycle of $PQ$ has this property. It follows that the minimal polynomial of $\mathfrak{m}(PQ)$ is precisely $\Phi_{2^{t}}(x)$.

Finally, we apply Proposition \ref{trick} to obtain that $\mathfrak{m}(P\mathcal{S}_{n}) = \Phi_{2^{t+1}}(x)$. This completes the proof.
\end{proof}

Theorem \ref{thm:hadexist} follows immediately from Proposition \ref{method}.

\begin{example}\label{ex:2}
Let $t=3$, so adhering to the procedure above, $a = (0,1,1)^{\top}$, $b = (1,1,1)^{\top}$, $u = (1,0,0)^{\top}$, and $v = (0,0,1)^{\top}$.  The matrix $\mathcal{S}_{3}$ is constructed under a lexicographic labeling of rows and columns; writing $-$ for $-1$,
\[
PQ =
\left[\begin{array}{cccccccc}
0 & 0 & 0 & 0 & 0 & 0 & 0 & 1\\
1 & 0 & 0 & 0 & 0 & 0 & 0 & 0\\
0 & 1 & 0 & 0 & 0 & 0 & 0 & 0\\
0 & 0 & 0 & 0 & 0 & 0 & 1 & 0\\
0 & 0 & 0 & - & 0 & 0 & 0 & 0\\
0 & 0 & 0 & 0 & - & 0 & 0 & 0\\
0 & 0 & 0 & 0 & 0 & - & 0 & 0\\
0 & 0 & - & 0 & 0 & 0 & 0 & 0\end{array}\right];
\]
and the matrices $D_{a}$, $T_{b}$, and $\rho(A)$ are, in order,
\[
\left[\begin{array}{cccccccc}
1 & 0 & 0 & 0 & 0 & 0 & 0 & 0\\
0 & - & 0 & 0 & 0 & 0 & 0 & 0\\
0 & 0 & - & 0 & 0 & 0 & 0 & 0\\
0 & 0 & 0 & 1 & 0 & 0 & 0 & 0\\
0 & 0 & 0 & 0 & 1 & 0 & 0 & 0\\
0 & 0 & 0 & 0 & 0 & - & 0 & 0\\
0 & 0 & 0 & 0 & 0 & 0 & - & 0\\
0 & 0 & 0 & 0 & 0 & 0 & 0 & 1\end{array}\right],
\left[\begin{array}{cccccccc}
0 & 0 & 0 & 0 & 0 & 0 & 0 & 1\\
0 & 0 & 0 & 0 & 0 & 0 & 1 & 0\\
0 & 0 & 0 & 0 & 0 & 1 & 0 & 0\\
0 & 0 & 0 & 0 & 1 & 0 & 0 & 0\\
0 & 0 & 0 & 1 & 0 & 0 & 0 & 0\\
0 & 0 & 1 & 0 & 0 & 0 & 0 & 0\\
0 & 1 & 0 & 0 & 0 & 0 & 0 & 0\\
1 & 0 & 0 & 0 & 0 & 0 & 0 & 0\end{array}\right],\]
and
\[
\left[\begin{array}{cccccccc}
1 & 0 & 0 & 0 & 0 & 0 & 0 & 0\\
0 & 0 & 0 & 0 & 0 & 0 & 0 & 1\\
0 & 0 & 1 & 0 & 0 & 0 & 0 & 0\\
0 & 0 & 0 & 0 & 0 & 1 & 0 & 0\\
0 & 1 & 0 & 0 & 0 & 0 & 0 & 0\\
0 & 0 & 0 & 0 & 0 & 0 & 1 & 0\\
0 & 0 & 0 & 1 & 0 & 0 & 0 & 0\\
0 & 0 & 0 & 0 & 1 & 0 & 0 & 0\end{array}\right].
\]
Then
\[
P\mathcal{S}_{3} =
\left[\begin{array}{cccccccc}
1 & 1 & 1 & 1 & - & - & - & -\\
- & 1 & 1 & - & - & 1 & 1 & -\\
- & - & 1 & 1 & 1 & 1 & - & -\\
1 & - & 1 & - & 1 & - & 1 & -\\
1 & - & 1 & - & - & 1 & - & 1\\
- & - & 1 & 1 & - & - & 1 & 1\\
- & 1 & 1 & - & 1 & - & - & 1\\
1 & 1 & 1 & 1 & 1 & 1 & 1 & 1\end{array}\right]
\]
is the constructed Hadamard matrix such that $8^{-1/2}P\mathcal{S}_{3}$ has characteristic polynomial $\Phi_{16}(x) = x^8 + 1$.  Hence if $H \in \mathrm{BH}(n,8)$, then with reference to Proposition \ref{trick} we have that $(\zeta_{16}H,P\mathcal{S}_{3})$ is a sound pair, and by Theorem \ref{thm:main} we have constructed a complete morphism $\mathrm{BH}(n,8) \rightarrow \mathrm{BH}(8n,2)$.
\end{example}

\subsection*{Acknowledgements}

The first author has been fully supported by the Croatian Science Foundation under the project 1637 and by the Irish Research Council (Government of Ireland Postdoctoral Fellowship, GOIPD/2018/304).  The authors also thank Ryan Vinroot for pointing out reference \cite{FulmanGuralnick}.  The authors acknowledge the helpful comments of the anonymous referees, which improved the exposition of the paper.

\end{document}